\documentclass[11pt]{birkjour}

 \usepackage{graphicx,color,amssymb, amsmath, amsthm, mathrsfs}
\usepackage{geometry}
\usepackage[all]{xy}
\usepackage{slashed}
\usepackage{mathtools}
\usepackage{todonotes}
\usepackage{color}
\usepackage[colorlinks = true,
            linkcolor = blue,
            urlcolor  = blue,
            citecolor = red,
            anchorcolor = blue]{hyperref}

\swapnumbers
\newtheorem{theorem}{Theorem}[section]
\newtheorem{lemma}[theorem]{Lemma}
\newtheorem{proposition}[theorem]{Proposition}
\newtheorem{corollary}[theorem]{Corollary}
\theoremstyle{definition}
\newtheorem{definition}[theorem]{Definition}

\newtheorem{remark}[theorem]{Remark}

\newtheorem{example}[theorem]{Example}
\newtheorem*{remark*}{Remark}
\newtheorem*{definition*}{Definition}

\usepackage{thmtools}
\declaretheoremstyle[notefont=\bfseries,notebraces={}{},%
    headpunct={},postheadspace=1em]{mystyle}
\declaretheorem[style=mystyle,numbered=no,name=Theorem]{thm-hand}

\newcommand{\Z}{\mathbf{Z}}

\newcommand{\C}{\mathbf{C}}

\newcommand{\bea}          {\begin{eqnarray}}
\newcommand{\eea}          {\end{eqnarray}}
\newcommand{\beastar}          {\begin{eqnarray*}}
\newcommand{\eeastar}          {\end{eqnarray*}}

\newcommand{\supp}{\textnormal{supp }}

\begin{document}

\date{\today}
\title{A $K$-theoretic Selberg trace formula}

\author[Mesland]{Bram Mesland}
\address{\normalfont{Mathematisch Instituut, Universiteit Leiden, \\
Postbus 9512, 2300 RA Leiden, Netherlands}}
\email{brammesland@gmail.com}

\author[\c{S}eng\"un]{Mehmet Haluk \c{S}eng\"un}
\address{\normalfont{School of Mathematics and Statistics,\\
University of Sheffield,
Hicks Building,\\
Hounsfield Road, Sheffield, S3 7RH, UK}}
\email{m.sengun@sheffield.ac.uk}

\author[Wang]{Hang Wang}
\address{\normalfont{Research Center for Operator Algebras, School of Mathematical Sciences\\
East China Normal University, Science building A505, 3663 North Zhongshan Rd,\\
Shanghai 200062, P.R.China}}
\email{wanghang@math.ecnu.edu.cn}

\begin{abstract} Let $G$ be a semisimple Lie group and $\Gamma$ a uniform lattice in $G$. The Selberg trace formula is an equality arising from computing in two different ways the traces of convolution operators on the Hilbert space $L^2(\Gamma \backslash G)$ associated to test functions $f \in C_c(G)$.

In this paper we present a cohomological interpretation of the trace formula involving the K-theory of the maximal group $C^*$-algebras of $G$ and $\Gamma$. As an application, we exploit the role of group $C^*$-algebras as recipients of ``higher indices" of elliptic differential operators and we obtain the index theoretic version of the Selberg trace formula developed by  Barbasch and Moscovici from ours.

\end{abstract}

\thanks{B. Mesland and M.H. \c{S}eng\"un gratefully acknowledge support from the Max Planck Institute for Mathematics in Bonn, Germany. H. Wang is supported by NSFC-11801178 and Shanghai Rising-Star Program 19QA1403200.}

\maketitle

\section{Introduction}
Let $G$ be a semisimple real Lie group and $\Gamma$ a cocompact lattice in $G$. A compactly supported smooth function on $G$ (a.k.a. a test function) acts on the Hilbert space $L^2(G/\Gamma)$ via convolution and it is well-known the bounded operator obtained from this action is of trace class. By computing this trace in two different ways, by ``summing the eigenvalues'' and by ``summing the diagonal entries", we arrive at Selberg's trace formula which has the following rough shape (see Section \ref{STF} for the precise formula):
$$ \sum_{\pi} m_\Gamma(\pi) \ {\rm tr}_\pi(f) = \sum_{(\gamma)} \mathcal{O}_{\gamma}(f).$$
The sum on the left hand side (the {\em spectral side}) runs over the irreducible unitary representations of $G$ and we sum, counting multiplicities, the trace of our operator on ``eigenspaces". The sum on the right hand side (the {\em geometric side}) runs over the conjugacy classes of $\Gamma$ and we sum over the {\em orbital integrals} which arise from the ``diagonal entries" of the kernel function associated to our operator.

The trace formula has fundamental applications in number theory. For example, the multiplicities $m_\Gamma(\pi)$ give dimensions of various spaces of automorphic forms. If $\pi$ is integrable, by plugging a carefully chosen test function (a {\em pseudo-coefficient}) into the trace formula, one can isolate $m_\Gamma(\pi)$ on the spectral side and thus get an explicit formula for it via the geometric side. Another important application is to the functoriality principle in the Langlands programme. For this application, one compares trace formulas on different groups.

In this paper, we present a cohomological carnation of the trace formula. We do this by considering the $K$-theory groups of the group $C^*$-algebras associated to $G$ and $\Gamma$, and interpreting the terms that appear in the trace formula in $K$-theoretic terms. Our beginning point is a simple commutative diagramme that is a $K$-theoretic manifestation of the fact that the quasi-regular representation $R^\Gamma$ of $G$ on $L^2(G/\Gamma)$ is the induction to $G$ of the trivial representation ${\bf 1}$ of $\Gamma$:
$$
\xymatrix@+2pc{K_0(C^*(G)) \ar[r]^{R^\Gamma_*} \ar[d]_{{\rm res}} & \Z \\
                     K_0(C^*(\Gamma)) \ar[ur]_{\mathbf{1}_*}&            }
$$
The homomorphism $R^\Gamma_*$ (resp. $\mathbf{1}_*$) arises via functoriality from $R_\Gamma$ (resp. $\mathbf{1}$) and the vertical arrow is the restriction homomorphism given by the Mackey-Rieffel theory of imprimitivity bimodules. We then use the spectral decomposition of  $R^\Gamma$ into irreducible unitary representations $\pi$ to decompose $R^\Gamma_*$ into homomorphisms $\pi_*$, thus obtaining the spectral side of our K-theoretic trace formula. For the geometric side, we decompose $1_*$ over the conjugacy classes of $\Gamma$. The decomposition starts at the level of the convolution algebra $L^1(\Gamma)$ of integrable functions on $\Gamma$. At this level, the partitioning of $\Gamma$ into its conjugacy classes $(\gamma)$ gives rise to a decomposition of the *-homomorphism ${\bf 1} : L^1(\Gamma) \to \C$ into maps $\tau_\gamma$ satisfying the trace property (i.e.  degree 0 cyclic cohains). Just for this introduction, let us assume that $\Gamma$ has ``controlled growth" (e.g. Gromov hyperbolic). Then it is possible to continuously extend the maps $\tau_\gamma$ to a holomorphically closed dense subalgebra of $C^*(\Gamma)$ and thus obtain a decomposition of ${\bf 1}_*$ into $\tau_{\gamma, *}$ at the level of $K_0(C^*(\Gamma))$.

We arrive at the following $K$-theoretic trace formula:
\begin{equation} \label{main-result}
\sum_{\pi\in\widehat G}m_{\Gamma}(\pi)\pi_*(x)=\sum_{(\gamma)\in\langle\Gamma\rangle}\tau_{\gamma, *}(\mathrm{res}(x)).
\end{equation}
for every $x \in K_0(C^*(G))$. The general result, which does not make any assumptions on $\Gamma$ unlike we did above, is given in Theorem \ref{thm:KSTF}.

An important feature of group $C^*$-algebras is that they are the recipients of the so-called {\em higher indices}. Let $K$ denote a maximal compact subgroup of $G$. Let $D$ denote the Dirac operator on the symmetric space $G/K$, possibly twisted by a finite dimensional representation of $K$ and denote by $\mathrm{Ind}_GD$ the higher index of $D$ over $C^*(G)$. We show that if one takes $x=[\mathrm{Ind}_GD]$ in Equation (\ref{main-result}), then we obtain, crucially using results of \cite{WW16}, the index theoretic version of the Selberg trace formula developed by Barbasch and Moscovici \cite{BM83}:
\begin{equation}
\label{eq:SelbergTraceFormulaHeatOp}
\sum_{\pi\in\widehat G}m_{\Gamma}(\pi)\mathrm{Tr}_s\pi(k_t)=\sum_{(\gamma)\in\langle\Gamma\rangle} \int_{\Gamma_{\gamma}\backslash G}\mathrm{Tr}_s k_t(x^{-1}\gamma x)dx.
\end{equation}
Here $t>0$ and $k_t$ is the kernel associated to the heat operator $e^{-tD^2}$ viewed as a matrix valued smooth function on $G$ with $\mathrm{Tr}_s$ denoting the supertrace.

From a larger perspective, our result allows for the use of techniques from the representation theory of the semisimple Lie group $G$ in the study of the $K$-theory of the  lattice $\Gamma$. It is part of a program that explores the use of operator algebras and $K$-theory in the theory of automorphic forms. See \cite{MS-1, MS-2} for other such results.

\section{Statement of Selberg trace formula} \label{STF}
We will now review the trace formula of Selberg in the setting of cocompact lattices in Lie groups (see, for example, \cite{Ar:Clay05} for a detailed general account). Let $G$ be a semisimple Lie group, $K$ a maximal compact subgroup and $\Gamma$ a uniform lattice in $G$.
Let $f\in C_c^{\infty}(G)$ and denote by $R^\Gamma: G\rightarrow \mathcal{U}(L^2(\Gamma\backslash G))$ the right regular representation which extends to $R^\Gamma: C_c^{\infty}(G)\rightarrow\mathcal{L}(L^2(\Gamma\backslash G))$ via
\[
R^\Gamma(f)=\int_G f(y)R^\Gamma(y)dy \qquad f\in C_c^{\infty}(G).
\]
The operator $R^\Gamma(f)$ is trace class (see below) and the Selberg trace formula is an equality arising from computing the trace of $R^\Gamma(f)$ in two different ways.

On the ``geometric side", regarding $\phi\in L^2(\Gamma\backslash G)$ as a $\Gamma$-invariant function on $G$, and from
\begin{align*}
[R^\Gamma(f)\phi](x)=&\int_Gf(y)[R^\Gamma(y)\phi](x)dy
=\int_G f(y)\phi(xy)dy\\
=&\int_G f(x^{-1}y)\phi(y)dy
=\int_{\Gamma\backslash G}\sum_{\gamma\in\Gamma}f(x^{-1}\gamma y)\phi(y)dy,
\end{align*}
we obtain the Schwartz kernel
\[
K(x,y)=\sum_{\gamma\in\Gamma}f(x^{-1}\gamma y),
\]
associated to $f$. The fact that $f$ is compactly supported and smooth ensures that the sum is locally finite and thus $K$ is a smooth kernel. Coupled with the fact that 
$\Gamma\backslash G$ is compact, we deduce that $R^\Gamma(f)$ is trace class. We have

\begin{align*}
\mathrm{Tr} R^\Gamma(f)=&\int_{\Gamma\backslash G}K(x, x)dx=\int_{\Gamma\backslash G}\sum_{\gamma\in\Gamma}f(x^{-1}\gamma x)dx\\
=&\int_{\Gamma\backslash G}\sum_{(\gamma) \in\langle\Gamma\rangle}\sum_{\delta\in\Gamma_{\gamma}\backslash \Gamma}f(x^{-1}\delta^{-1}\gamma\delta x)dx\\
=&\sum_{(\gamma) \in\langle \Gamma\rangle}\mathrm{vol}(\Gamma_{\gamma}\backslash G_{\gamma})\int_{G_{\gamma}\backslash G}f(x^{-1}\gamma x)dx.
\end{align*}
As usual, in the above $\langle \Gamma\rangle$ denotes the conjugacy classes of $\Gamma$ and $\Gamma_\gamma$ (resp. $G_\gamma$) is the centralizer of $\gamma$
 in $\Gamma$ (resp. $G$).

On the ``spectral side", the Hilbert space $L^2(\Gamma \backslash G)$ splits as a direct sum
 \begin{equation}
 \label{irrepdecomp}
 L^2(\Gamma\backslash G)\simeq\bigoplus_{\pi\in\widehat G}H_{\pi}^{\oplus m_{\Gamma}(\pi)}
 \end{equation}
 of irreducible unitary representations $(\pi, H_\pi) \in \widehat G$ of $G$ with each appearing with finite multiplicity $m_{\Gamma}(\pi)$. 
By restricting $R^\Gamma(f)$ to the irreducible subspaces of $L^2(\Gamma\backslash G)$, we obtain the following from \eqref{irrepdecomp}, 
\[
\mathrm{Tr}R^\Gamma(f)=\sum_{\pi\in\widehat G} m_{\Gamma}(\pi)\mathrm{Tr}(\pi(f)).
\]
where $\pi(f)=\int_G f(y)\pi(y)dy$ (which can be regarded as the Fourier transform of $f$ at $\pi$).
The Selberg trace formula is the equality
\begin{equation}
\label{eq:SelbergTraceFormula}
\sum_{\pi\in\widehat G}m_{\Gamma}(\pi)\mathrm{Tr}(\pi(f))=\sum_{(\gamma) \in\langle \Gamma\rangle}\mathrm{vol}(\Gamma_{\gamma}\backslash G_{\gamma})\int_{G_{\gamma}\backslash G}f(x^{-1}\gamma x)dx.
\end{equation}

\begin{remark}
If $\Gamma\backslash G$ has finite volume but is noncompact, the spectral decomposition of $L^2(\Gamma\backslash G)$ also involves a continuous component and $R^{\Gamma}(f)$ is not necessarily a trace class operator. A truncated version of traces need to be introduced in the Selberg trace formula in this setting (see \cite{Ar:Clay05}).
\end{remark}

\section{Index theoretic trace formula}
In this section we formulate an analogue of the Selberg trace formula in the context of index theory, which will be proved using the framework of $K$-theory in Section~\ref{sec:K-theoreticSelbergTraceFormula}.

We {\bf assume} in addition that the rank of $G$ equals that of $K$. Note that as a consequence, the symmetric space $G/K$ is of even dimension.
Assume that $G/K$ admits a $G$-equivariant spin$^c$ structure and let $S=S^+ \oplus S^-$ denote the associated spinor bundle on $G/K$. Let $V$ be a finite dimensional representation of $K$ and $E=G\times_K V$ be the associated $G$-equivariant vector bundle on $G/K$. Let $D$ be the $G$-invariant Dirac operator on the bundle $G\times_K (V\otimes S)\rightarrow G/K$ with $\Z_2$-grading, i.e., $D$ has the form $\begin{bmatrix}0 & D^-\\ D^+ & 0\end{bmatrix}$. Consider the heat operator $e^{-tD^2}$, which admits a  smooth kernel $K_t(x, y)$. The kernel $K_{t}(x,y)$ satisfies
\[
K_t(gx, gy)=K_t(x, y) \qquad x, y\in G/K, \, g\in G,
\]
which is equivalent to the $G$-invariance of $D$.
Regarding $K_t(x, y)$ as a matrix-valued function on $G\times G$, invariant under the $K\times K$-action
\[
K_t(x, y)=aK_t(xa, yb)b^{-1}\in\mathrm{End}(V\otimes S) \qquad a, b\in K.
\]
We define
\[
k_t(x^{-1}y):=K_t(x, y)\qquad x, y\in G.
\]
Then $k_t$ is a matrix-valued function on $G$ that is invariant under the action of $K\times K$:
\[
k_t\in[C^{\infty}(G)\otimes\mathrm{End}(V\otimes S)]^{K\times K}, \qquad  k_t(x)=ak_t(a^{-1}xb)b^{-1}, \quad a,b\in K.
\]
Moreover, $k_t$ is smooth and of Schwartz type. To be more precise,
\[
k_t\in [\mathcal{S}(G)\otimes\mathrm{End}(V\otimes S)]^{K\times K}.
\]
where $\mathcal{S}(G)$ is Harish-Chandra's  Schwartz algebra of $G$.
See~{\cite[Section 1]{CM82}} and~{\cite[Section 2]{BM83}} for more details.

We now replace the test function $f\in C_c^{\infty}(G)$ acting on $L^2(\Gamma \backslash G)$ in the previous section with the heat kernel $k_t$ acting on the $\Z_2$-graded space $(L^2(\Gamma \backslash G) \otimes V \otimes S)^K$ and formally state the Selberg trace formula associated to the heat operator. We write $\pi(k_t)$ for the  operator obtained from the action of $k_t$ on
$(H_\pi \otimes (V\otimes S))^K$.

Choose an invariant Haar measure on $G$ so that it is compactible with $G$-invariant measure on $G/K$. This in particular means that the volume of the maximal compact subgroup $K$ of $G$ is assumed to be $1.$

\begin{proposition} Let $\mathrm{Tr}_s$ be the supertrace of $\Z_2$-graded vector spaces.
 Then
\begin{equation}
\label{eq:SelbergTraceFormulaHeatOp}
\sum_{\pi\in\widehat G}m_{\Gamma}(\pi)\mathrm{Tr}_s\pi(k_t)=\sum_{(\gamma) \in\langle\Gamma\rangle}\mathrm{vol}(\Gamma_{\gamma}\backslash G_{\gamma})\int_{G_{\gamma}\backslash G}\mathrm{Tr}_s [k_t(x^{-1}\gamma x)\gamma]dx.
\end{equation}
\end{proposition}

The heat kernel $k_t$ is not compactly supported, however it is of $L^p$-Schwartz class for all $p>0$ (see \cite[Prop. 2.4.]{BM83}). Using this, one can prove that each summand on the right hand side of equality (\ref{eq:SelbergTraceFormulaHeatOp}) converges, see \cite[Thm. 6.2]{WW16}. In any case, the equality (\ref{eq:SelbergTraceFormulaHeatOp}) will follow from our main result Theorem~\ref{thm:mainresult} below.

In the rest of this section, we reformulate the equality (\ref{eq:SelbergTraceFormulaHeatOp}) in terms of an equality of indices of Dirac operators.

\subsection{Geometric side}

In this subsection we will see that every term on the geometric side can be identified as the $L^2$-Lefschetz number associated to a conjugacy class of $\Gamma$ using results from~\cite{WW16}, and that the sum of all $L^2$-Lefschetz numbers over all conjugacy classes $\langle\Gamma\rangle$ can be identified with the Kawasaki index theorem~\cite{Kawasaki:1981} for the orbifold $\Gamma\backslash G/ K.$

Because $\Gamma$ acts properly and  cocompactly on $G/K$, there exists a cutoff function $c$ on $G/K$ with respect to the action of $\Gamma$, i.e., $c: G/K\rightarrow[0, 1]$ such that
$$\sum_{\gamma\in\Gamma}c(\gamma x)=1\qquad  \forall x\in G/K.$$
Define the $L^2$-Lefschetz number associated to the conjugacy class $(\gamma)$ of $\gamma\in\Gamma$ by
$$\mathrm{ind}_{\gamma}D:=\mathrm{tr}_s^{(\gamma)}e^{-tD^2}
=\mathrm{tr}^{(\gamma)}e^{-tD^-D^+}-\mathrm{tr}^{(\gamma)}e^{-tD^+D^-}$$
where
$$\mathrm{tr}^{(\gamma)} S:=\sum_{h\in(\gamma)}\int_{G/K}c(x)\mathrm{Tr}[h^{-1}K_S(hx, x)]dx,$$
for any $(\gamma)$-trace class operator $S$ with smooth Schwartz kernel $K_S$, see (3.22) of~\cite{WW16}.
By Theorems~3.23 and 6.1 of \cite{WW16}, $\mathrm{ind}_{\gamma}D$ admits a fixed point formula: for $h\in (\gamma)$ the kernel $h^{-1}K_{t}(hx, h)$ localizes to the submanifold $(G/K)^{\gamma}$ consisting of points in $G/K$ fixed by $\gamma$.
By \cite[Theorem 6.2]{WW16}, the $L^2$-Lefschetz number for $G/K$ admits the following expression analogous to the orbital integral on the geometric side of \eqref{eq:SelbergTraceFormula}:
\begin{equation}
\label{eq:loc.ind.ker}
\mathrm{ind}_{\gamma}D=\mathrm{vol}(\Gamma_{\gamma}\backslash G_{\gamma})\int_{G_{\gamma}\backslash G}\mathrm{Tr}_s [k_t(x^{-1}\gamma x)\gamma]dx.
\end{equation}

\begin{example}\label{ex:free}
If $\Gamma$ acts freely on $G/K$, then $(G/K)^{\gamma}$ is empty and all orbital integrals vanish except for the term where $\gamma$ is the group identity.
The geometric side of (\ref{eq:SelbergTraceFormulaHeatOp}) in this special case is equal to \mbox{$\mathrm{vol}(\Gamma\backslash G)\cdot \mathrm{ind}_eD$}.
Here $\mathrm{ind}_eD$ is a topological index called $L^2$-index (denoted  $\mathrm{ind}_{L^2}\, D$) for the symmetric space $G/K$ introduced in~\cite{CM82}.
\end{example}

Denote by $D_\Gamma$ the operator on $\Gamma\backslash G/K$ descending from the one on $G/K.$ It is a Dirac operator on the compact orbifold $\Gamma\backslash G/K$. Recall that by the Kawasaki orbifold index formula, the Fredholm index $\mathrm{ind}\, D_\Gamma$ can be decomposed as a sum over conjugacy classes of $\Gamma$ and the summands are exactly the $L^2$-Lefschetz numbers:
\begin{equation}
\label{eq:index.decomp.conj}
\mathrm{ind}\, D_\Gamma=\sum_{(\gamma) \in\langle\Gamma\rangle}\mathrm{ind}_{\gamma}D.
\end{equation}
See~\cite{Kawasaki:1981} and Section~2.2 and Theorem~6.2 of~\cite{WW16}. See also~\cite{Do} for the relationship between heat kernels on $X$ and the orbifold $\Gamma\backslash X.$

\begin{remark} In the special case of a free action, $\mathrm{ind}\, D_\Gamma$ is the Fredholm index of $D_\Gamma$ on the closed manifold $\Gamma\backslash G/K$.
By Atiyah's $L^2$-index theorem~\cite{Atiyah:1976}, we obtain
\[
\mathrm{vol}(\Gamma\backslash G/K) \, \mathrm{ind}_{L^2}\, D=\mathrm{ind}\, D_\Gamma,
\]
which is a special case of~(\ref{eq:index.decomp.conj}). Noting that $K$ is assumed to have volume $1$, this formula is compatible with Example~\ref{ex:free}.
\end{remark}

\subsection{Spectral side}

Recall the Plancherel decomposition of $L^2(G)$:
\[
L^2(G)=\int_{\widehat G}^{\oplus}(H_{\pi}^*\otimes H_{\pi})d\mu(\pi),
\]
Accordingly, the Dirac operator $D$ on $[L^2(G)\otimes V\otimes S]^K$ has a decomposition into a family of Dirac type operators $D_{\pi}$ on $[H_{\pi}\otimes V\otimes S]^K$,
parameterized by irreducible unitary representations $(\pi, H_{\pi})\in\widehat G$.

We recall the definition and properties of $D_{\pi}$ introduced in~\cite{Moscovici82} and Section 7 of~\cite{CM82}. See also~\cite{Huang-Pandzic} for a comprehensive treatment.
Let $H_{\pi}^{\infty}$ be the space of $C^{\infty}$-vectors for the $G$-representation $\pi$.
The $G$-invariant operator $D$ can be written as a finite sum $D=\sum_i R^\Gamma(X_i)\otimes A_i$ where $X_i$ belongs to the universal enveloping algebra of $\mathfrak g_{\C}$ and $A_i^{\pm}\in\mathrm{Hom}(V\otimes S^{\pm}, V\otimes S^{\mp})$.
Here $R^\Gamma$ stands for the right regular representation.
Then $D_{\pi}$ is given by
\[
D_{\pi}: [H^{\infty}_{\pi}\otimes V\otimes S]^K\rightarrow [H^{\infty}_{\pi}\otimes V\otimes S]^K,\quad D_{\pi}=\sum_i \pi(X_i)\otimes A_i.
\]
It is proved in~\cite{CM82,Moscovici82} that $D_{\pi}$ is essentially self-adjoint on the dense domain $[H^{\infty}_{\pi}\otimes V\otimes S]^K$. Its closure (also denoted $D_{\pi}$) is a Fredholm operator whose Fredholm index
 $$\mathrm{ind} \, D_{\pi}=\mathrm{dim}(\mathrm{ker} \, D_{\pi}^+)-\mathrm{dim}(\mathrm{ker} \, D_{\pi}^-)$$
 is equal to $\mathrm{Tr}_s\pi(k_t)$ on the spectral side of (\ref{eq:SelbergTraceFormulaHeatOp}). Indeed, by Proposition~2.1 of~\cite{BM83},
  one has
 $$\mathrm{Tr}(\pi(k_t^{\pm}))=\mathrm{Tr}(e^{-tD_{\pi}^{\mp}D_{\pi}^{\pm}}).$$
  Then by the McKean-Singer formula, we obtain
\begin{align}
\nonumber \mathrm{ind}\, D_{\pi}
  &=\mathrm{Tr}(e^{-tD_{\pi}^{-}D_{\pi}^{+}})-\mathrm{Tr}(e^{-tD_{\pi}^{+}D_{\pi}^{-}})\\
 \label{eq:ind.Dpi} &=\mathrm{Tr}(\pi(k^+_t))-\mathrm{Tr}(\pi(k_t^-))
  =\mathrm{Tr}_s(\pi(k_t)).
\end{align}

\begin{remark}
For an irreducible representation $V$ of $K$, the dimensions of $[H_{\pi}\otimes V\otimes S^+]^K$ and $[H_{\pi}\otimes V\otimes S^-]^K$ are finite (see~\cite{Atiyah-Schmid, CM82, BM83}). The index of the Dirac operator $D_{\pi}$ corresponding to $\pi$ can be calculated as the difference
\[
\mathrm{ind}\, D_{\pi}=\mathrm{Tr}_s\pi(k_t)=\mathrm{dim}[H_{\pi}\otimes V\otimes S^+]^K-\mathrm{dim}[H_{\pi}\otimes V\otimes S^-]^K.
\]
\end{remark}

In view of~(\ref{eq:loc.ind.ker}) and~(\ref{eq:ind.Dpi}), we conclude the following equivalent statement of the index theoretic Selberg trace formula.
\begin{proposition}
\label{prop:index.th.STF}
The equality~(\ref{eq:SelbergTraceFormulaHeatOp}) is equivalent to the equality of indices:
\[
\sum_{\pi\in\widehat G}m_{\Gamma}(\pi) \,\mathrm{ind}\, D_{\pi}=\sum_{(\gamma) \in\langle\Gamma\rangle}\mathrm{ind}_{\gamma}D.
\]
The left hand side is finite sum (see Corollary \ref{lem:resrep} below) and the convergence of the right hand side is determined by~(\ref{eq:index.decomp.conj}).
\end{proposition}

\begin{remark}Let $(\eta, H_{\eta})$ be a discrete series representation, and $\mu\in\mathfrak t^*_{\C}$ the Harish-Chandra parameter of $\eta$. There is a standard way, due to Atiyah-Schmid~\cite{Atiyah-Schmid}, to realize $H_{\eta}$ as the $L^2$-kernel of a Dirac type operator on $G/K$. For an irreducible representation $V_{\nu}\in R(K)$ of $K$ with highest weight $\nu\in\mathfrak t^*_{\C},$ let $D^{\nu}$ denote the associated Dirac operator defined on the homogeneous vector bundle
\[
G\times_K(V_{\nu}\otimes S)\rightarrow G/K.
\]
Denoting by $\rho_c$ the half sum of compact positive roots, then $H_{\eta}$ can be identified with the $L^2$-kernel of $D^{\mu+\rho_c}.$
It is also proved in \cite[Section 4]{Atiyah-Schmid} that
\[
\mathrm{dim}(H_{\eta}^*\otimes V_{\nu}\otimes S^{\pm})^K=0,
\]
unless $\mu=\nu+\rho_c$. From this, one can derive Proposition~\ref{prop:index.th.STF} when $D=D^{\eta}$.
See also~\cite{Baum-Connes-Higson} for some concrete examples.
\end{remark}

\begin{example} \label{discrete-series-example}
 Let $(\eta, H_{\eta})\in\widehat G$ be a discrete series with Harish-Chandra paramater $\mu$. Let $\pi$ be an arbitrary discrete series. 
 It follows from~\cite{Atiyah-Schmid} that
\[
\mathrm{ind} (D^{\mu+\rho_c})_{\pi}=\mathrm{dim}[H_{\pi}^*\otimes V_{\mu+\rho_c}\otimes S^{+}]^K-\mathrm{dim}[H_{\pi}^*\otimes V_{\mu+\rho_c}\otimes S^{-}]^K=\begin{cases}1 & \eta=\pi \\ 0 & \eta\neq\pi.\end{cases}
\]
Assume further that $\Gamma$ is a torsion-free. Then the index theoretic Selberg trace formula (\ref{eq:SelbergTraceFormulaHeatOp}) reduces to
\[
m_{\Gamma}(\eta)=\mathrm{vol}(\Gamma\backslash G/K) \mathrm{ind}_{L^2}\, D^{\mu+\rho_c},
\]
recovering a main result of Pierrot~\cite{Pierrot}, which is a special case of~\cite{Langlands}.
\end{example}

\section{K-theoretic Selberg Trace Formula}
\label{sec:K-theoreticSelbergTraceFormula}

\subsection{Decomposition of right regular representation of $G$}

Let us consider again the right regular representation $R^\Gamma$ of $G$ on $L^2(\Gamma\backslash G)$. As mentioned earlier, $C^\infty_{c}(G)$ acts on $L^2(\Gamma\backslash G)$ by compact operators. We use the standard notation $\mathbb{K}(X)$ for the algebra of compact operators on a Hilbert $C^{*}$-module $X$. As $C^\infty_{c}(G)$ is a dense subalgebra of the maximal group $C^*$-algebra $C^*(G)$ of $G$, we obtain a $*$-homomorphism
\[
R^{\Gamma}:C^*(G)\rightarrow\mathbb{K}(L^2(\Gamma\backslash G)).
\]

This induces a homomorphism on $K$-theory:
\[
R^\Gamma_*: K_0(C^*(G))\rightarrow K_0(\mathbb{K}(L^2(\Gamma\backslash G))\simeq \Z.
\]
As $G$ is liminal, the above observation applies to any irreducible unitary representation as well. Indeed, for any $(\pi, H_{\pi}) \in \widehat G$, we have homomorphisms
\[
\pi: C^*(G)\rightarrow\mathbb{K}(H_{\pi}),\qquad  \pi_*: K_0(C^*(G))\rightarrow \Z.
\]
The morphism $\pi_*$ determines an element in $KK(C^*(G), \C)$ which we will denote $[\pi].$
Observe that we have
$$\pi_*(x)=x\otimes_{C^*(G)}[\pi],$$
for every $x\in K_0(C^*(G))$, where $\otimes_{C^*(G)}$ denotes the Kasparov product. We remark that the maps $\pi_*$ feature\footnote{Lafforgue uses the notation $\langle H, x \rangle$ for our $\pi_*(x)$, where $H$ is the representation space of $\pi$.} heavily in Chapter 2 of \cite{LafforgueICM}.

\begin{proposition}
\label{locally-finite}
For every $x \in K_0(C^*(G))$ there exists a finite subset $S_{x} \subset \widehat G$ such that
\begin{enumerate}
\item $\pi_*(x)=0$ for all $\pi \in \widehat G \setminus S_x$,
\item $R^\Gamma_*(x) = \sum\limits_{\pi \in S_x} \pi_*(x)$.
\end{enumerate}
\end{proposition}

\begin{proof}
Let $x\in K_{0}(C^{*}(G))$ be given by the even $(\C,C^{*}(G))$ Kasparov module $(X,F)$. Since $1\in\C$ acts on $X$ by a projection, say $p$, after replacing $X$ by $pX$ and $F$ by $pFp$, we may without loss of generality assume that $\C$ acts unitally on $X$ and $F^{2}-1\in\mathbb{K}(X)$. Since $C^{*}(G)$ acts on $L^{2}(\Gamma\setminus G)$ by compact operators, we have that $\mathbb{K}(X)\otimes 1\subset \mathbb{K}(X\otimes_{R^{\Gamma}} L^{2}(\Gamma\setminus G))$ (see \cite[Proposition 4.7]{Lance}). The  operator
\[F\otimes 1:X\otimes_{R^{\Gamma}}L^{2}(\Gamma\setminus G)\to X\otimes_{R^{\Gamma}}L^{2}(\Gamma\setminus G),\]
satisfies $(F\otimes 1)^{2}-1=(F^{2}-1)\otimes 1\in \mathbb{K}(X)\otimes 1\subset \mathbb{K}(X\otimes L^{2}(\Gamma\setminus G))$. Thus $F\otimes 1$ is a self-adjoint unitary modulo compact operators, and therefore it is Fredholm. The integer $R^{\Gamma}_{*}(x)\in\mathbb{Z}$ is thus gives by the index of $F\otimes 1$.
Since there is a direct sum splitting
\[X\otimes_{R^{\Gamma}}L^{2}(\Gamma\setminus G)\simeq \bigoplus_{\pi\in\widehat{G}} \left(X\otimes_{\pi} H_{\pi}\right)^{\oplus m_{\Gamma}(\pi)},\]
it follows that $F\otimes 1=\bigoplus_{\pi\in\widehat{G}} F_{\pi}$ where
the operators are defined to be
\[F_{\pi}:=F\otimes_{\pi}1:\left(X\otimes_{\pi} H_{\pi}\right)^{\oplus m_{\Gamma}(\pi)}\to \left(X\otimes_{\pi} H_{\pi}\right)^{\oplus m_{\Gamma}(\pi)}.\]
Since $F\otimes 1$ is Fredholm, its kernel and cokernel are finite dimensional and thus it follows that $F_{\pi}$ is unitary for all but finitely many $\pi\in\widehat{G}$. Therefore the set
\[S_x:=\{\pi\in\widehat{G}: \pi_{*}(x)=\textnormal{Ind }(F_{\pi})\neq 0\},\]
is finite and the restriction
\[F_{\widehat{G}\setminus S_x}:= F\otimes 1:\bigoplus_{\pi\in \widehat{G}\setminus S_x}\left(X\otimes_{\pi} H_{\pi}\right)^{\oplus m_{\Gamma}(\pi)}\to \bigoplus_{\pi\in \widehat{G}\setminus S_x}\left(X\otimes_{\pi} H_{\pi}\right)^{\oplus m_{\Gamma}(\pi)},\]
is Fredholm of index $0$. Moreover
\[F_{S_x}:\bigoplus_{\pi\in  S_x}\left(X\otimes_{\pi} H_{\pi}\right)^{\oplus m_{\Gamma}(\pi)}\to \bigoplus_{\pi\in  S_x}\left(X\otimes_{\pi} H_{\pi}\right)^{\oplus m_{\Gamma}(\pi)},\]
obviously satisfies $\textnormal{Ind }(F_{S_x})=\sum\limits_{ \pi\in S_x} \pi_{*}(x)$. Since
\[R^\Gamma_*(x) =\textnormal{Ind }(F\otimes 1)=\textnormal{Ind } (F_{S_x})= \sum\limits_{\pi \in S_x} \pi_*(x),\]
the result follows.
\end{proof}
For the quasi-regular representation, the above proposition gives a decomposition of $R^\Gamma_*$.
\begin{corollary} \label{R_*-decomposition} For $x \in K_0(C^*(G))$, we have
$$R^\Gamma_*(x)=\sum_{\pi\in\widehat G}m_{\Gamma}(\pi) \, \pi_*(x),$$
understood as a sum with only finitely many nonzero terms.
\end{corollary}
As is well-known, $K_0(C^*(G))$ is the recipient of the so-called higher indices which we recall now. A Dirac type operator $D$ on $G/K$ has a $G$-index
$$\mathrm{Ind}_{G}(D):=[p]\otimes_{C_0(G/K)\rtimes G}j^{\Gamma}([D])\in K_0(C^*(G))$$
defined by the descent map $j_{\Gamma}: K^*_{\Gamma}(C_0(X))\rightarrow KK(C_0(G/K)\rtimes G, C^*(G))$ followed by a compression on the left by the projection
$p\in C_0(G/K)\rtimes G$ given by
$$p(x,g):=c(gx)c(x), \quad x \in G/K, g \in G$$
where $c: G/K \rightarrow [0,\infty)$ is a continuous compactly supported function satisfying $\int_G c(s^{-1}x)^2ds=1$ for all $x\in G/K$.  See, for example, \cite[Section 4.2]{Echterhoff} for details.

\begin{corollary} \label{lem:resrep} Let $D$ be a $G$-invariant Dirac type operator on $G/K$ with index $\mathrm{Ind}_GD\in K_0(C^*(G))$. Then
\begin{enumerate}
\item[(i)] Given $\pi \in \widehat G$, we have $ \pi_*(\mathrm{Ind}_GD)=\mathrm{ind}\, D_{\pi},$
\item[(ii)] $\mathrm{ind}\, D_{\pi} \not= 0$ for at most finitely many $\pi \in \widehat G$.
\end{enumerate}
\end{corollary}

\begin{remark}
In Fox-Haskell~\cite{Fox-Haskell}, the example of $G=Spin(4,1)$, its unitary dual and K-theory is explicitly studied. The paper showed that for a generator $[f]\in K_0(C^*(G)),$ discrete series, reducible principle series, trivial representations and endpoint representations could appear in the support of $f$. Moreover, if $\pi\in\hat G$ appears in the support of $f$, then $\pi_*([f])=\pm1$ which is nonvanishing.
\end{remark}

\begin{remark}
If $G$ has property (T), then the trivial representation ${\bf 1}$ gives rise to a generator $[{\bf 1}]=[(\C, 0)]\in K_0(C^*(G))$. Then we have  $R^\Gamma_*([{\bf 1}])={\bf 1}_*([{\bf 1}])=1$ since 
$$\dim(\C\otimes_{C^*(G)}L^2(\Gamma\backslash G))=\dim(L^2(G)^G)=1.$$
\end{remark}

It is possible to obtain refined information about the localized indices using representation theory.  We have already remarked in Example \ref{discrete-series-example} that the index of $D_\pi$ can be nonzero when $\pi$ is in the discrete series. The next proposition is Prop. 3 of \cite{Labesse} (see also \cite[Prop.~7.3]{CM82} and \cite[Lem.~1.1]{Deitmar} for related results).
\begin{proposition} Let $\pi \in \widehat G$ be tempered. If $\pi$ is not in the discrete series nor a limit of discrete series, then the index of $D_\pi$ is zero.
\end{proposition}

\begin{remark} \label{Kazhdan}
Recall that there is a surjective morphism $\lambda: C^*(G)\rightarrow C^*_r(G)$ and the Connes-Kasparov isomorphism $R(K)\simeq K_0(C^*_r(G))$ assuming existence of $G$-equivariant spin$^c$-structure of $G/K$. The Connes-Kasparov isomorphism is defined through  Dirac induction
$$R(K)\simeq K_0(C^*_r(G)), \qquad [V_{\mu}]\mapsto\mathrm{Ind}_{G, r}(D^{V_{\mu}}).$$
Here $V_{\mu}$ is an irreducible unitary representation of $K$ with highest weight $\mu$ and $D^{V_{\mu}}$ is the Dirac operator on $G/K$ twisted by the associated bundle $G\times_{K}V_{\mu}\rightarrow G/K.$ The twisted Dirac operator $D^{V_{\mu}}$ has an equivariant index $\mathrm{Ind}_G(D^{V_{\mu}})\in K_0(C^*(G))$ so that $\lambda_*(\mathrm{Ind}_G(D^{V_{\mu}}))=\mathrm{Ind}_{G,r}(D^{V_{\mu}})$.
That is, there is a commutative diagram
\[
\xymatrix{R(K) \ar[r]^{\mathrm{Ind}_G} \ar[dr]_{\mathrm{Ind}_{G,r}} & K_0(C^*(G))\ar[d]^{\lambda_*} \\
                 &    K_0(C^*_r(G)).            }
\]
If $G$ is not $K$-amenable, i.e., $\lambda_*$ is not an isomorphism, then not all $x\in K_0(C^*(G))$ are represented by the equivariant index of some elliptic operator.
For example, if $G$ has property (T), the Kazhdan projection defines a nontrivial element $[p]\in K_0(C^*(G))$ but $\lambda_*[p]=0\in K_0(C^*_r(G)).$

We will discuss a special case of Corollary~\ref{R_*-decomposition} where
$$x=\mathrm{Ind}_G(D^{V_{\mu}})\in K_*(C^*(G)).$$
Recall that
$$R_*(\mathrm{Ind}_G(D^{V_{\mu}}))=\mathbf{1}_{*}(\mathrm{Ind}_{\Gamma}(D^{V_{\mu}}))=\mathrm{ind} D^{V_{\mu}}_{\Gamma}$$
where $D^{V_{\mu}}_{\Gamma}$ is the operator $D^{V_{\mu}}$ descended to $\Gamma\backslash G/K$.
Then by \cite[Section 3]{Moscovici82} and \cite[(1.2.4)]{BM83}
$$\mathrm{ind} D^{V_{\mu}}_{\Gamma}=\sum_{\pi\in\widehat{G}}m_{\Gamma}(\pi)\mathrm{ind}(D^{V_{\mu}})_{\pi}.$$
Furthermore, the summand on the right hand side is nonvanishing only when the infinitesimal character $\chi_{\pi}$ of $\pi\in\widehat{G}$ coincides with that of $\mu+\rho_c$, hence it is a finite sum. See \cite{Atiyah-Schmid} and \cite[(1.3.7)-(1.3.8)]{BM83}.
In summary, we obtain
$$R_*(x)=R_*(\mathrm{Ind}_G(D^{V_{\mu}}))=\sum_{\pi\in\widehat{G}, \chi_{\pi}=\chi_{\mu+\rho_c}}m_{\Gamma}(\pi)\mathrm{ind}(D^{V_{\mu}})_{\pi}.$$

A typical example of such $x$ is represented by the integrable discrete series. The integrable discrete series representations of $G$ appear as isolated points in $\widehat G$ (see \cite{Wang}). As a result, they give rise to generators $[p_\pi]$ for $K_0(C^*(G))$.
Recall that every discrete series representation of $G$ gives rise to a generator of $K_0(C^*_r(G))$ (see \cite[Thm. 4.6]{Rosenberg}).
So $\lambda_*[p_{\pi}]=[p_{\pi}]\in K_0(C^*_r(G)).$
If $\pi$ is an integrable representation, it follows from Example \ref{discrete-series-example} and Corollary \ref{R_*-decomposition} that
\[
R^\Gamma_*([p_\pi])=m_{\Gamma}(\pi).
\]

Let $x\in K_0(C^*(G))$ now be arbitrary. Then there exists finitely many $\lambda_{\mu}\in\widehat{K}$ with multiplicity $m_{\lambda_{\mu}}$ such that
\[
\lambda_*(x)=\sum_{\lambda_{\mu}\in\widehat{K}}m_{\lambda_{\mu}}\mathrm{Ind}_{G,r}D^{V_{\mu}}.
\]
Let $x_0=\sum_{\lambda_{\mu}\in\widehat{K}}m_{\lambda_{\mu}}\mathrm{Ind}_G D^{V_{\mu}}.$
By definition
\[\lambda_*(x-x_0)=0,\quad\textnormal{ and} \quad R_*(x)=R_*(x_0)+R_*(x-x_0),\]
where $R_*(x-x_0)$ does not come from index theory and
$R_*(x_0)$ can be calculated by
\[
R_*(x_0)=\sum_{\lambda_{\mu}\in\widehat{K}}R_*(\mathrm{Ind}_G D^{V_{\mu}})=\sum_{\pi\in\widehat{G}}m_{\Gamma}(\pi)\sum_{\substack{\lambda_{\mu}\in\widehat{K}, \\ \chi_{\pi}=\chi_{\mu+\rho_c}}} m_{\lambda_{\mu}}\mathrm{ind}(D^{V_{\mu}})_{\pi}.
\]
Thus if we were to consider $x_0$ in
$K_0(C^*_r(G))$, then $x_0$ can be decomposed into a finite sum involving elliptic operators on $G/K$. Note however that we cannot work with $C^*_r(G)$ in this paper, as one knows that non-tempered representations of $G$ may enter $R^\Gamma$ \cite[p.177]{Wallach}.
\end{remark}

\subsection{Decomposition of the trivial representation of $\Gamma$}

Consider a uniform lattice $\Gamma$ in $G$ and denote by $\mathbf{1}$ the trivial representation of $\Gamma$.
The corresponding representation of $C^*(\Gamma)$ is a $*$-homomorphism and hence a trace
\[
\mathbf{1}: C^*(\Gamma)\rightarrow \C \qquad \sum_{\gamma\in\Gamma}a_{\gamma}\gamma\mapsto \sum_{\gamma\in\Gamma}a_{\gamma}.
\]
Thus there is a morphism of $K$-theory:
\[
\mathbf{1}_*: K_0(C^*(\Gamma))\rightarrow \Z.
\]
Let $\gamma \in \Gamma$ with conjugacy class $(\gamma)$. Then we have a well-defined localized trace on the Banach subalgebra  $L^1(\Gamma)\subset C^*(\Gamma)$,
\[
\tau_{\gamma}: L^1(\Gamma)\rightarrow \C, \qquad \sum_{g \in\Gamma} a_g g \mapsto \sum_{g \in (\gamma)}a_g.
\]
For $a, b\in L^1(\Gamma)$ we have the tracial property $\tau_{\gamma}(a*b)=\tau_{\gamma}(b*a)$ with respect to the convolution product $*$. This implies the existence of the morphism
\[
\tau_{\gamma,*}: K_0(L^1(\Gamma))\rightarrow \C.
\]
Let $\iota: L^1(\Gamma)\rightarrow C^*(\Gamma)$ be the inclusion map. It is straightforward to observe that
\begin{equation}
\label{eq:dec.1.conj}
\sum_{\gamma\in\langle\Gamma\rangle} \tau_{\gamma,*}=\textbf{1}_*\iota_*: K_0(L^1(\Gamma))\rightarrow \Z.
\end{equation}

Given a Dirac type operator $D$ on $G/K$, we can form its $\Gamma$-index $\mathrm{Ind}_{\Gamma}(D)$ which lands in $K_0(C^*(\Gamma))$ following the recipe we outlined right before Corollary \ref{lem:resrep}. If we choose a properly supported parametrix, the index class can be presented in a smaller algebra  $\C\Gamma \otimes \mathcal{R}$ where
$\mathcal{R}$ is the algebra of operators with smooth kernels and compact support. In particular, the $\Gamma$-index homomorphism $\mathrm{Ind}_\Gamma$ factors through
the $L^1$-index homomorphism  $K^{\Gamma}_0(G/K)\xrightarrow{\mathrm{Ind}_{\Gamma, L^1}}  K_*(L^1(\Gamma))$. That is, there is $\mathrm{Ind}_{\Gamma, L^1}(D) \in K_0(L^1(\Gamma))$  such that
$$\iota_*(\mathrm{Ind}_{\Gamma, L^1}(D))=\mathrm{Ind}_{\Gamma}(D).$$
The following lemma is proved in Definition 5.7 and Proposition 5.9 in~\cite{WW16}.

\begin{lemma}
\label{lem:local.ind.K}
\[
\tau_{\gamma,*}(\mathrm{Ind}_{\Gamma, L^1}(D))=\mathrm{ind}_{\gamma}(D).
\]
\end{lemma}

Then together with (\ref{eq:dec.1.conj})  we have
\[
{\bf 1}_*(\mathrm{Ind}_{\Gamma}(D))=\sum_{(\gamma) \in\langle\Gamma\rangle} \mathrm{ind}_{\gamma}(D).
\]
Let $D_{\Gamma}$ be the operator on the quotient orbifold $\Gamma\backslash G/K$ induced by $D$. Note that ${\bf 1}_*(\mathrm{Ind}_{\Gamma}(D))$ is equal to the Fredholm index of $D_{\Gamma}$ (see \cite{Bunke}).
Therefore, we have an equality of indices obtained from an identity on $K$-theory:
\[
\mathrm{ind}(D_{\Gamma})=\sum_{(\gamma) \in\langle\Gamma\rangle} \mathrm{ind}_{\gamma}(D).
\]
The reader should compare this to (\ref{eq:index.decomp.conj}).

\subsection{The restriction map}
Let $\Gamma\subset G$ be a discrete subgroup and consider the restriction map
$$\rho :C_{c}(G)\to C_{c}(\Gamma)\subset C^{*}(\Gamma),$$
defined by $\rho(f)(\gamma):=f(\gamma)$. The map $\rho$ defines a positive definite $C^{*}(\Gamma)$ valued inner product on $C_{c}(G)$ by
$$\langle f,h\rangle (\gamma):=\rho (f^{*}*h)(\gamma)=\int_{G}f^{*}(\xi)g(\xi^{-1}\gamma)d\mu_{G}(\xi).$$
Upon completion, a $(C^{*}(G),C^{*}(\Gamma))$-$C^{*}$-bimodule, denoted ${\sf Res}$, where the action of $C^{*}(G)$ is induced by convolution on $C_{c}(G)$.  Denote by $[g]$ the class of $g$ in $G/\Gamma$.
The commutative $C^{*}$-algebra $C_{0}(G/\Gamma)$ acts on $C_{c}(G)$ from the left via
$$(f\cdot \phi)(g):=f([g])\phi(g),$$
and this extends to an action by adjointable endomorphisms of ${\sf Res}$. Elements of $C_{0}(G/\Gamma)$ do not act compactly unless $G$ is discrete.

By Rieffel's imprimitivity theorem (see \cite{Rieffel}) the algebra of compact endomorphisms of ${\sf Res}$ is
$$\mathbb{K}({\sf Res})=C^{*}(G\ltimes G/\Gamma).$$
The dense subalgebra $C_{c}(G\ltimes G/\Gamma)$ acts on the dense submodule $C_{c}(G)\subset {\sf Res}$ via
$$(K\cdot \phi)(g):=\int_{G}K(\xi, [\xi^{-1}g])\phi(\xi^{-1}g)d\mu_{G}\xi.$$

\begin{lemma}
\label{cpt}
Let $f\in C_{c}(G)$ and $h\in C_{0}(G/\Gamma)$. Then $(f\cdot h)(g,x):=f(g)h(x)$ is an element of $C_{c}(G\ltimes G/\Gamma)$ whose action on $C_{c}(G)$ is given by
\begin{equation}\label{action}
(f\cdot h)\phi(g)=\int_{G}f(\xi)h([\xi^{-1}g])\phi(\xi^{-1}g)d\mu_{G}\xi=\int_{G}f(\xi)(h\cdot \phi)(\xi^{-1}g)d\mu_{G}\xi.\end{equation}
Consequently we have that $C^{*}(G)\cdot C_{0}(G/\Gamma)\subset \mathbb{K}({\sf Res})$.
\end{lemma}
\begin{proof} The function $f\cdot h$ satisfies $\supp (f\cdot h)\subset \supp f\times \supp h$ and therefore has compact support, that is it is an element of
$C_{c}(G\ltimes G/\Gamma)$. The remaining claims now follow by direct calculation.
\end{proof}

\begin{proposition}
\label{prop:restriction}
 Let $\Gamma\subset G$ be a discrete cocompact group and ${\sf Res}$ the associated $(C^{*}(G),C^{*}(\Gamma))$-bimodule defined above. Then $C^{*}(G)$ acts on ${\sf Res}$ by compact module endomorphisms. Hence the bimodule ${\sf Res}$ defines an element $[{\sf Res}]\in KK_{0}(C^{*}(G), C^{*}(\Gamma))$.
\end{proposition}
\begin{proof} As $G/\Gamma$ is compact we have that $C_{0}(G/\Gamma)=C(G/\Gamma)$ and the constant function $1\in C_{0}(G/\Gamma)$. By Lemma \ref{cpt} it follows that $f\cdot 1\in \mathbb{K}({\sf Res})$. It follows from Equation \eqref{cpt} that the action of $f$ on $C_{c}(G)$ coincides with that of $f\cdot 1$ on $C_{c}(G)$ and thus $C_{c}(G)$ acts by compact endomorphisms. Hence by continuity all of $C^{*}(G)$ acts compactly. The remaining claims now follow directly.
\end{proof}

\begin{definition}
\label{def:res}
The restriction map
\[
\mathrm{res}: K_0(C^*(G))\rightarrow K_0(C^*(\Gamma))
\]
is given by the KK-product with ${\sf Res} \in KK_0(C^*(G), C^*(\Gamma))$ in Proposition~\ref{prop:restriction}:
\[
\mathrm{res}([x]):=[x]\otimes_{C^*(G)} [{\sf Res}].
\]
\end{definition}
Denote by ${\bf 1}: C^*(\Gamma)\rightarrow\C$ the $*$-homomorphism induced by the trivial representation of $\Gamma$ on $\C$:
\[
{\bf 1}:f\mapsto\sum_{\gamma\in\Gamma}f(\gamma), \qquad f\in C_c(\Gamma),
\]
and by $[\mathbf{1}]\in KK_{0}(C^{*}(\Gamma),\C)$ its $K$-homology class.
\begin{proposition}
\label{prop:comm}
We have a commutative diagram
\begin{equation}
\label{eq:res.com.}
\xymatrix@+2pc{K_*(C^*(G)) \ar[r]^{\ \ \ \ -\otimes[R^\Gamma]} \ar[d]_{{\rm res}} & \Z \\
                     K_*(C^*(\Gamma)) \ar[ur]_{\ \ \ -\otimes [\mathbf{1}]}&            }
\end{equation}
Here $[R^\Gamma]$ is the element in $K^0(C^*(G))$ determined by the morphism $R^\Gamma_*$.
\end{proposition}
\begin{proof}
Let $(X,F)$ be a $(\C,C^{*}(G))$ Kasparov module with $\C$ acting unitally (as in the proof of Proposition \ref{locally-finite}).
Then by definition \[[(X,F)]\otimes_{C^*(G)}[R^\Gamma] = R^\Gamma_*([(X,F)])\in\Z,\] is given by the index of the Fredholm operator 
\[
F\otimes 1: X\otimes_{R^{\Gamma}} L^{2}(\Gamma\setminus G)\to X\otimes_{R^{\Gamma}} L^{2}(\Gamma\setminus G).
\]
On the other hand,
\begin{align*}
{\rm res}([(X,F)]) \otimes_{C^*(\Gamma)} [\mathbf{1}]&= [(X,F)] \otimes_{C^*(G)} [{\sf Res}] \otimes_{C^*(\Gamma)} [\mathbf{1}] \\
&=[(X\otimes_{C^{*}(G)}{\sf Res} \otimes_{C^*(\Gamma)}\C, F\otimes 1\otimes 1)].
\end{align*}
By Proposition \ref{prop:restriction} we have $C^{*}(G)\to \mathbb{K}({\sf Res})$ and the trivial representation $\mathbf{1}$ is finite dimensional. Thus a successive application of \cite[Proposition 4.7]{Lance} shows that
\[(F\otimes 1\otimes 1)^{2}-1=(F^{2}-1)\otimes 1\otimes 1\in \mathbb{K}(X)\otimes 1\subset \mathbb{K}(X\otimes_{C^{*}(G)}{\sf Res} \otimes_{C^*(\Gamma)}\C),\]
proving that $F\otimes 1\otimes 1$ is Fredholm. Therefore ${\rm res}([(X,F)]) \otimes_{C^*(\Gamma)} [\mathbf{1}]\in\Z$ equals the index of $F\otimes 1\otimes 1$ and to show the diagram commutes, it is sufficient to show that there is a unitary isomorphism
\begin{equation}
\label{eq:isom}
\Phi: {\sf Res} \otimes_{C^*(\Gamma)}\C\xrightarrow{\sim} L^2(\Gamma\setminus G),
\end{equation}
of Hilbert spaces intertwining the $C^{*}(G)$-representations. 
Define a map 
\[
\Phi: C_{c}(G)\otimes_{C_{c}(\Gamma)}\C\to C( \Gamma\setminus G),\qquad \Phi(f)(\Gamma g)=\sum_{\gamma\in\Gamma}f(g^{-1}\gamma).
\]
Note that the above map is well defined, i.e., the right hand sides coincide for $\Gamma g=\Gamma g'.$
Then on one hand
\begin{align*}
\langle \Phi(f), \Phi(h) \rangle=&\int_{G/\Gamma}\overline{\Phi(f)(\Gamma g)}\Phi(h)(\Gamma g){\rm d}(\Gamma g) \\
=&\int_{G/\Gamma}\sum_{\gamma\in\Gamma}\overline{f(g^{-1}\gamma)}\sum_{\delta\in\Gamma}h(g^{-1}\delta){\rm d}(\Gamma g)\\
=&\sum_{\delta\in\Gamma}\int_G\overline{f(g^{-1})}h(g^{-1}\delta){\rm d}g,
\end{align*}
whereas on the other hand 
\[
\langle f\otimes 1, g\otimes 1\rangle=\sum_{\delta\in\Gamma}\langle f, h\rangle_{C^*(\Gamma)}(\delta)=\sum_{\delta\in\Gamma}\int_G\overline{f(g^{-1})}h(g^{-1}\delta){\rm d}g=\langle \Phi(f), \Phi(h) \rangle.
\]
Hence the map $\Phi$ is an isometry and extends to an injection. Since the action of $\Gamma$ on $G$ is properly discontinuous, functions of small support on $\Gamma\setminus G$ are in the image of $\Phi$, and a partition of unity argument shows that $\Phi$ is surjective, so it extends to a unitary isomorphism. For $\xi\in G$ we have
\[\xi(\Phi(f))(\Gamma g)=\Phi(f)(\Gamma g\xi)=\sum_{\gamma\in\Gamma} f(\xi^{-1}g^{-1}\gamma)=\sum_{\gamma\in\Gamma} (\xi f)(g^{-1}\gamma) =\Phi(\xi f)(\Gamma g),\]
hence $\Phi$ intertwines the $G$-representations and therefore the $C^{*}(G)$-  representations. Thus we proved the isomorphism~(\ref{eq:isom}). 
\end{proof}

\subsection{Selberg trace formula in K-theory}

We are ready for the $K$-theoretic Selberg trace formula. Let $\iota: L^1(\Gamma)\rightarrow C^*(\Gamma)$ be the natural inclusion and $\iota_*: K_0(L^1(\Gamma))\rightarrow K_0(C^*(\Gamma))$ be the associated map. Putting the preceding discussions together, we obtain the following.
\begin{theorem}\label{thm:KSTF} Let $x \in K_0(C^*(G))$. Assume that there exists $y\in K_0(L^1(\Gamma))$ such that $\iota_*(y)=\mathrm{res}(x) \in K_0(C^*(\Gamma))$. Then
$$\boxed{\sum_{\pi\in\widehat G}m_{\Gamma}(\pi)\pi_*(x)=\sum_{(\gamma) \in\langle\Gamma\rangle}\tau_{\gamma,*}(y).}$$
\end{theorem}


\begin{remark}

We can drop the hypothesis on $x$ by putting conditions on $\Gamma$. For example, if $\Gamma$ admits the polynomial growth condition of each of its conjugacy classes (see~\cite{S}), or if $\Gamma$ is a word hyperbolic group (see~\cite{Puschnigg}), then the trace $\tau_{\gamma}: L^1(\Gamma) \rightarrow\C$ can be extended to a subalgebra of $C^*(\Gamma)$ that is stable under holomorphic functional calculus and thus gives rise to a well-defined morphism
\[
\tau_{\gamma, *}: K_0(C^*(\Gamma))\rightarrow\C.
\]
It then follows that
$$\sum_{\pi\in\widehat G}m_{\Gamma}(\pi)\pi_*(x)=\sum_{(\gamma) \in\langle\Gamma\rangle}\tau_{\gamma,*}({\rm res}\, (x))$$
for {\em all} $x\in K_0(C^*(\Gamma))$.

\end{remark}

With a little more work, we can derive a general index theoretic trace formula from the above. For this, we need to bring in the equivariant $K$-homology groups into the picture.

\begin{lemma}
\label{lem:comm}
The following diagram commutes:
$$\xymatrix@+2pc{K^G_0(G/K)\ar[r]^{\mathrm{Ind}_G}\ar[d]_{\mathfrak{r}} & K_*(C^*(G))  \ar[d]_{{\rm res}}  \\
                   K^{\Gamma}_0(G/K)\ar[r]^{\mathrm{Ind}_{\Gamma}} & K_*(C^*(\Gamma)) &            }$$
where the left vertical arrow $\mathfrak{r}$ is the restriction map in equivariant $KK$-theory.
\end{lemma}

\begin{proof}
Every cycle representing an element of $K^G_0(G/K)$ is given by $([L^2(G)\otimes V]^K, F)$ where $V$ is a $\Z_2$-graded finite dimensional representation of $K$, and $F$ is a bounded properly supported $G$-invariant operator on the Hilbert space $H=[L^2(G)\otimes V]^K.$ Its images under $\mathrm{Ind}_{G}$ and under $\mathrm{Ind}_{\Gamma} \circ \mathfrak{r}$ are represented by cycles of the form $(\mathcal{E}_{G}, F)$, $(\mathcal{E}_{\Gamma}, F)$.
Here the Hilbert $C^*(G)$-module $\mathcal{E}_G$ (resp. $C^*(\Gamma)$-module $\mathcal{E}_{\Gamma}$) is given by completion of $[C_c(G)\otimes V]^K$ with respect to the $C_c(G)$-valued ($C_c(\Gamma)$-valued) inner product determined by
\[
\langle f_1, f_2\rangle(t)=\int_G \langle f_1(s), tf_2(t^{-1}s)\rangle_{H}ds\qquad  f_1, f_2\in H
\]
for $t\in G$ ($t\in \Gamma$).
Recall that ${\sf Res}$ is the closure of $C_c(G).$
To show that
$$[(\mathcal{E}_{G}, F)]\otimes_{C^*(G)}[{\sf Res}]=[(\mathcal{E}_{\Gamma}, F)],$$
we only need to observe that $\mathcal{E}_{\Gamma}= \mathcal{E}_{G}\otimes_{C^*(G)} {\sf Res}$. This follows directly from the isomorphism of $C_c(\Gamma)$-modules
\[
[C_c(G)\otimes V]^K\otimes_{C_c(G)}C_c(G)_{C_c(\Gamma)}\simeq[C_c(G)_{C_c(\Gamma)}\otimes V]^K,
\]
which is compatible with the inner products.
\end{proof}

By Theorem~\ref{thm:KSTF} and Lemmas~\ref{lem:resrep}, \ref{lem:local.ind.K} and~\ref{lem:comm} we obtain the following theorem, giving the Selberg trace formula in its index theory form.

\begin{theorem}
\label{thm:mainresult}
The following diagram commutes
$$\xymatrix@+2pc{K^G_0(G/K)\ar[r]^{\mathrm{Ind}_G}\ar[d]_{\mathfrak{r}} & K_*(C^*(G)) \ar[dr]^{R^\Gamma_*=\sum\limits_{\pi\in\widehat G}m_{\Gamma}(\pi)\pi_*} \ar[d]_{{\rm res}} &  \\
                   K^{\Gamma}_0(G/K)\ar[r]_{\mathrm{Ind}_{\Gamma}}\ar[dr]_{\mathrm{Ind}_{\Gamma, L^1}}  & K_*(C^*(\Gamma)) \ar[r]_{\mathbf{1}_*}&            \Z \\
                   & K_0(L^1(\Gamma))\ar[u]^{\iota_*}\ar[ur]_{\sum\limits_{(\gamma) \in\langle \Gamma\rangle}\tau_{\gamma, *}} }.$$
In particular, if $D$ is a $G$-invariant Dirac type operator on $G/K$, the above commutative diagram gives rise to the equality
\[
\sum_{\pi\in\widehat G}m_{\Gamma}(\pi)\pi_*(\mathrm{Ind}_G D)=\sum_{(\gamma) \in\langle\Gamma\rangle}\tau_{\gamma, *}(\mathrm{Ind}_{\Gamma, L^1}D),
\]
which reduces to the index theoretic Selberg trace formula:
\begin{equation}
\label{eq:IndSTF}
\boxed{\sum_{\pi\in\widehat G}m_{\Gamma}(\pi)\mathrm{ind}D_{\pi}=\sum_{(\gamma) \in\langle\Gamma\rangle}\mathrm{ind}_{\gamma}D.}
\end{equation}
\end{theorem}

\begin{remark}
Denote by $D_{\Gamma}$ the Dirac type operator on $\Gamma\backslash G/K$  whose lift to $G/K$ is $D$. Then
\[
\mathrm{ind}\, D_{\Gamma}=\sum_{(\gamma) \in\langle\Gamma\rangle}\mathrm{ind}_{\gamma}D.
\]
Also $[\mathrm{Ind}_GD]\otimes_{C^*(G)}[R^\Gamma]=R^\Gamma_*(\mathrm{Ind}_GD)=\sum_{\pi\in\widehat G}m_{\Gamma}(\pi)\mathrm{ind}D_{\pi}.$
Then~(\ref{eq:IndSTF}) is reduced to
\[
[\mathrm{Ind}_GD]\otimes_{C^*(G)}[R^\Gamma]=[D_{\Gamma}]\in KK(\C, \C)=\Z.
\]
This recovers Theorem~2.3 in~\cite{Fox-Haskell}.
\end{remark}


\end{document}